\documentclass{amsart}

\title[A Geometric Interpretation of Higher-Order Hankel Operators]{A Geometric Interpretation and Explicit Form for Higher-Order Hankel Operators}
\author{Benjamin Pittman-Polletta}
\date{July, 2008}

\usepackage{mathdots,extarrows}
\usepackage[margin=1in]{geometry}

\newtheorem{thm}{Theorem}[section]

\newtheorem{prop}[thm]{Proposition}
\newtheorem{lem}[thm]{Lemma}
\theoremstyle{definition}
\newtheorem{dfn}[thm]{Definition}

\DeclareSymbolFont{AMSb}{U}{msb}{m}{n}
\DeclareMathSymbol{\N}{\mathbin}{AMSb}{"4E}
\DeclareMathSymbol{\Z}{\mathbin}{AMSb}{"5A}
\DeclareMathSymbol{\R}{\mathbin}{AMSb}{"52}
\DeclareMathSymbol{\Q}{\mathbin}{AMSb}{"51}
\DeclareMathSymbol{\I}{\mathbin}{AMSb}{"49}
\DeclareMathSymbol{\C}{\mathbin}{AMSb}{"43}

\newcommand{\recip}[1]{\frac{1}{#1}}

\newcommand{\oo}{\infty}

\newcommand{\paren}[1]{\left(#1\right)}
\newcommand{\parsh}[1]{\frac{\partial}{\partial#1}}

\newcommand{\pplus}{\mathcal{P}_+}
\newcommand{\pminus}{\mathcal{P}_-}
\newcommand{\delt}{\Delta}
\newcommand{\linear}{\mathcal{L}_2\left(\ahol{1/2},\hol{1/2}\right)}
\newcommand{\algebra}{\mathfrak{sl}(2,\C)}
\newcommand{\nchoosek}[2]{\left(\!\!\begin{tabular}{c} $#1$ \\ $#2$ \end{tabular}\!\!\right)}
\newcommand{\bigo}{\mathcal{O}}
\newcommand{\tenslow}{\pi_{\otimes}\left(A^-\right)}

\newcommand{\hol}[1]{H^{#1}_{L^2}(\delt)}
\newcommand{\ahol}[1]{H^{#1}_{L^2}(\delt^*)}
\newcommand{\dz}[1]{(dz)^{#1}}
\newcommand{\dzh}{(dz)^{1/2}}
\newcommand{\phol}[1]{H^{#1}_{poly}(\delt)}

\newcommand{\pz}[1]{\left(\parsh{z}\right)^{#1}}
\newcommand{\ssym}{S^s\left(H^{-1}(\delt)\right)}
\newcommand{\opow}[1]{^{\odot #1}}
\newcommand{\symraise}{\pi_{\odot}\left(A^+\right)}
\newcommand{\symlow}{\pi_{\odot}\left(A^-\right)}
\newcommand{\symham}{\pi_{\odot}\left(E\right)}

\begin{document}

\maketitle

\begin{abstract}
This paper deals with group-theoretic generalizations of classical Hankel operators called higher-order Hankel operators. We relate higher-order Hankel operators to the universal enveloping algebra of the Lie algebra of vector fields on the unit disk. From this novel perspective, higher-order Hankel operators are seen to be linear differential operators. An attractive combinatorial identity is used to find the exact form of these differential operators.
\end{abstract}

\section{Introduction}

A classical Hankel operator is a map between Hilbert spaces whose matrix representation is constant along antidiagonals. A survey of these operators and their applications appears in \cite{Peller}. Hankel operators arise naturally in the study of holomorphic function spaces. Given $f(z)=\sum_{j\in\Z}f_jz^j\in L^2\paren{S^1}$ and $x(z)=\sum_{j=1}^{\oo}x_jz^j$, define the operators
$$\pplus f(z)=\sum_{j=0}^{\oo}f_jz^j,\qquad M_xf(z)=x(z)f(z),\qquad \pminus f(z)=\sum_{j=1}^{\oo}f_{-j}z^{-j}.$$
The projection $\pplus$ is known as the Cauchy-Szeg\H{o} projection, and $\pplus L^2\paren{S^1}$ is the space of holomorphic functions on the open unit disk with square-integrable boundary values. The operator $B_1(x)=\pplus M_x\pminus$ is a Hankel operator, whose matrix representation we will write
$$B_1(x)=\left(\begin{matrix}
\vdots & & & \\
x_3 & & \iddots  & \\
x_2 & x_3 & & \\
x_1 & x_2 & x_3 & \ldots \\
\end{matrix} \right).$$
The function $x$ is called the symbol of $B_1(x)$, and the map $x\longmapsto B_1(x)$ is conformally equivariant in a sense explained in \S 2. Group-theoretic generalizations of this map, which are also conformally equivariant, were discovered in \cite{JP}. Here, we derive the following explicit expressions for these higher-order Hankel operators.
\begin{thm}
\label{main}
The higher-order Hankel operator of order $s+1$ with symbol $x(z)(dz)^{-s}$, where $x(z)=\sum_{j=s+1}^{\oo}x_jz^{s+j}$, has the formula
$$B_{s+1}(x)=\pplus L_s(x)\pminus,$$
where $L_s(x)$ is the differential operator
$$L_s(x)=\sum_{j=0}^s\recip{s!}\nchoosek{s}{j}\nchoosek{s+j}{j}x^{(s-j)}\left(\parsh{z}\right)^j.$$
\end{thm}
\noindent In the case $s=1$, this operator has the matrix representation
$$B_2(x)=\left(\begin{matrix} 
\vdots & & & & \\
4x_5 & & & \iddots & \\
3x_4 & 2x_5 & & & \\
2x_3 & x_4 & 0 &  & \\
x_2 & 0 & -x_4 & -2x_5 & \\
0 & -x_2 & -2x_3 & -3x_4 & -4x_5 & \ldots
\end{matrix} \right).$$

The higher-order Hankel forms introduced in \cite{JP} are related to the transvectants $\tau^j_{k,l}$ of classical invariant theory \cite{Gordan}, and to the Rankin-Cohen brackets appearing in the theory of modular forms \cite{R, Cohen}. Our map $B_{s+1}$ is the adjoint of the transvectant $\tau^s_{\recip{2},\recip{2}}$. These objects and their relationships have been studied from many perspectives \cite{EG, OS, P, PZ, RTY}. While the content of Theorem \ref{main} is basically known, our method of proof - viewing $B_{s+1}(x)$ as an element of a universal enveloping algebra - appears to have some novelty. 

The rest of the paper is as follows. In \S 2, we introduce spaces $\hol{m}$ of sections of line bundles which are isomorphic to weighted Bergman spaces, and a group action on them. In \S 3, we outline the proof of Theorem \ref{main}, which involves first showing that $B_{s+1}$ is a linear differential operator of order $\leq s$, and then determining its coefficients. In \S 4, we find embeddings of the spaces $\hol{-s}$ into the universal enveloping algebra of the Lie algebra of vector fields on the unit disk. These embeddings determine the form of $B_{s+1}$. In \S 5, we find the image under $B_{s+1}$ of a key element of $\hol{-s}$. In \S 6, we prove Theorem \ref{main}, using several identities for binomial coefficients. In \S 7, we present a pair of binomial coefficient identities resulting from Theorem \ref{main}, and in \S 8 we relate the maps $B_{s+1}$ to transvectants and higher-order Hankel forms.

\section{Notation}

The group 
$$G=PSU(1,1)=\left\{g=\pm\left(\begin{matrix} a&b\\ \bar {b}&\bar {a}\end{matrix} \right):\vert a\vert^2-\vert b\vert^2=1\right\}$$
acts on the Riemann sphere $\hat {\Bbb C}$ by linear fractional transformations, 
$$g:z\longmapsto\frac {\bar {b}+\bar {a}z}{a+bz}.$$
Let $\Delta$ and $\Delta^*$ denote the open unit disks around $0$ and infinity. The decomposition 
\begin{align}\label{decomp}
\hat{\C}=\Delta^{*}\sqcup S^1\sqcup\Delta,
\end{align}
is stable under the action of $G$. Restricting its action to $\Delta$ identifies $G$ with the group of conformal automorphisms of $\Delta$.

For each half-integer $m$, the action of $G$ on $\hat{\C}$ lifts to an action of $SU(1,1)$ on $\kappa^m$, the $m^{\mathrm{th}}$ (tensor) power of the canonical bundle.  This induces an action of $SU(1,1)$ on the space of holomorphic sections of $\kappa^m|_{\Delta}$, the holomorphic differentials of degree $m$ on $\Delta$.  We denote such a differential by $f(z)(dz)^m$, and the space of all such differentials by $H^m(\delt)$.  The action of $g\in SU(1,1)$ on $H^m(\delt)$ is
\begin{align}\label{action}
g:f(z)(dz)^m\longmapsto f\left(\frac {-\bar{b}+az}{\bar{a}-bz}\right)(\bar{a}-bz)^{-2m}(dz)^m.
\end{align}

The action of $SU(1,1)$ on $H^m(\delt)$ is essentially unitary for $m>0$; the dense Hilbert subspace of $H^m(\delt)$ with Hermitian inner product
\begin{align}\label{innerproduct}
\langle f(z)(dz)^m,g(z)(dz)^m\rangle_m=\begin{cases}
&\int_{\delt}f(z)\bar{g}(z)(1-|z|^2)^{2m-2}dzd\bar{z},\qquad m=1,\frac{3}{2},2,\ldots,\\
&\int_{S^1}f(z)\bar{g}(z)dz,\qquad m=\frac{1}{2}.
\end{cases}
\end{align}
will be denoted $\hol{m}$. Thus, $f(z)(dz)^m\in\hol{m}$ if and only if $f(z)$ belongs to the weighted Bergman space $A^2_{2m-2}(\delt)$.

The space of sections of $\kappa^{1/2}|_{S^1}$ will be denoted $\Omega^{1/2}\paren{S^1}$. The action of $SU(1,1)$ and the $SU(1,1)$-invariant Hermitian inner product on this space are also given by (\ref{action}) and (\ref{innerproduct}), with $m=1/2$. We will denote the Hilbert subspace by $\Omega^{1/2}_{L^2}\paren{S^1}$, so that $f(z)(dz)^{1/2}\in\Omega^{1/2}\paren{S^1}$ if and only if $f(z)\in L^2\paren{S^1}$.

The decomposition (\ref{decomp}) corresponds to an $SU(1,1)$-stable decomposition
\begin{align}\label{hardy}
\Omega^{1/2}_{L^2}\paren{S^1}=H^{1/2}_{L^2}(\Delta)\oplus H^{1/2}_{L^2}(\Delta^*).
\end{align}
We abuse notation slightly by using $\pplus$ to denote the projection onto $\hol{1/2}$, and $\pminus$ to denote the projection onto $\ahol{1/2}$. As noted, $\pplus f(z)(dz)^{1/2}=f_+(z)(dz)^{1/2}$, where $f_+(z)$ is the Cauchy-Szeg\H{o} projection (Cauchy transform) of $f(z)$. 

If $\theta=f(z)(dz)^{1/2}\in H^{1/2}_{L^2}$ and $\eta=g(z)(dz)^{1/2}\in H^{1/2}_{L^2}(\Delta^*)$, then $\theta\eta=fgdz$ is a one density on $S^1$ that can be integrated to a nontrivial constant, so $\left(H^{1/2}_{L^2}(\Delta)\right)^*=H^{1/2}_{L^2}(\Delta^*)$. The decomposition (\ref{hardy}) induces a diagonal action of $SU(1,1)$ on Hilbert-Schmidt operators sending $\ahol{1/2}$ to $\hol{1/2}$, via the identifications
$$\mathcal{L}_2\left(\ahol{1/2},\hol{1/2}\right)=\hol{1/2}\otimes\left(\ahol{1/2}\right)^*=\hol{1/2}\otimes\hol{1/2}.$$
Calculating the character of the rotation group $S^1\subset PSU(1,1)$, one sees that
$$\hol{1/2}\otimes\hol{1/2}=\hol{1}\oplus\hol{2}\oplus\hol{3}\oplus\ldots$$
so for each $m\in\N$ there is an intertwining map $\hol{m}\longrightarrow\linear$.

Let $x\in H^0(\delt)/\C$. Then $x$ acts on $\Omega^{1/2}(S^1)$ by multiplication. With respect to the decomposition (\ref{hardy}), the multiplication operator $M_x$ can be written
$$M_x=\left(\begin{matrix} 
A & B\\
C & D
\end{matrix} \right).$$
$B=B_1(x)=\pplus M_x\pminus\in\linear$ is the Hankel operator associated to $x$. The action of $G$ on $H^0(\delt)/\C$ intertwines with the diagonal action of $SU(1,1)$ on $\linear$, and $B_1(x)$ is Hilbert-Schmidt precisely when $\theta=x'dz\in\hol{1}$. Thus $B_1$ is an $SU(1,1)$-equivariant map from $\hol{1}$ to $\linear$. This motivates the following.

\begin{dfn}
The Hankel operator of order $m$ with symbol $\theta\in\hol{m}$ is the image of $\theta$ under the intertwining map
$$\hol{m}\longrightarrow\linear=\hol{1/2}\otimes\hol{1/2}.$$
The space of Hankel operators of order $m$ is the image of $\hol{m}$ under this map.
\end{dfn}

\section{An Outline of the Proof}

Our geometric interpretation of higher-order Hankel operators relies on two facts. First, for $m>0$, the action of $SU(1,1)$ on $H^m(\delt)$ is essentially irreducible. But if $m=-s$ for $s\in\N$, there is a short exact sequence
$$0\longrightarrow\C^{2s+1}=\mathrm{span}\left\{\dz{-s},\ldots,z^{2s}\dz{-s}\right\}\longrightarrow H^{-s}(\delt)\xlongrightarrow{I_s} H^{s+1}(\delt)\longrightarrow 0.$$
$H^{-s}(\delt)/\C^{2s+1}$ can be given a Hilbert space structure via the intertwining map $I_s$, and we refer to this Hilbert space as $\hol{-s}$. We study the maps $B_{s+1}:\hol{-s}\longrightarrow\hol{1/2}\otimes\hol{1/2}$.

Second, in the case $s=1$, $H^{-1}(\delt)$ is the Lie algebra of complex vector fields on the unit disk (note $\dz{-1}=\parsh{z}$) . These vector fields act on $\Omega^{1/2}(S^1)$ as differential operators of order $1$ via the Lie derivative. Composing this action with $\pplus,\pminus$ gives us a Hankel operator of order two. For $\theta=f(z)\dzh\in\ahol{1/2}$ and $x(z)\parsh{z}\in\hol{-1}$, a formal calculation leads to
$$B_2(x)\theta:=\pplus\recip{2}L_{x(z)\parsh{z}}f(z)\dzh=\pplus\left(\recip{2}x'f+xf'\right)\dzh.$$

The key to the higher-order Hankel operators is that the above action of $H^{-1}(\delt)$ determines an action of $H^{-s}(\delt)$ on $\Omega^{1/2}(S^1)$, as differential operators of order $\leq s$. Each representation of $H^{-1}(\delt)$ corresponds to a representation of the universal enveloping algebra $\mathcal{U}\left(H^{-1}(\delt)\right)$. The elements of $\mathcal{U}\left(H^{-1}(\delt)\right)$ act as linear differential operators on $\Omega^{1/2}(S^1)$. Let $S^s\left(H^{-1}(\delt)\right)$ be the space of symmetric tensors of order $\leq s$ over $H^{-1}(\delt)$. $S\left(H^{-1}(\delt)\right)=\bigoplus_{s=0}^{\oo}S^s\left(H^{-1}(\delt)\right)$ is isomorphic as a filtered vector space to $\mathcal{U}\left(H^{-1}(\delt)\right)$, so elements of $S^s\left(H^{-1}(\delt)\right)$ can be mapped to linear differential operators of order $\leq s$ on $\Omega^{1/2}(S^1)$. In \S 3, we show the existence of a $G$-equivariant embedding of $H^{-s}(\delt)$ into $S^s\left(H^{-1}(\delt)\right)$. Thus the image of $H^{-s}(\delt)$ under $B_{s+1}$ contains linear differential operators of order $\leq s$. 

$B_{s+1}$ is unique up to multiplication by a constant, and must map the lowest-weight vector in $H^{-s}(\delt)/\C^{2s+1}$, namely $z^{2s+1}\pz{s}$, to the lowest-weight vector $l_s$ of weight $2(s+1)$ in $\hol{1/2}\otimes\hol{1/2}$. We find the form of $l_s$ in \S 4, and use it to find $B_{s+1}$ in \S 5.

Finally, a technical point. The elements of $H^{-1}(\delt)$ may not extend to holomorphic vector fields on $S^1$, so neither $H^{-1}(\delt)$ nor $\mathcal{U}\left(H^{-1}(\delt)\right)$ act naturally on $\Omega^{1/2}(S^1)$ a priori. However, the polynomial sections of $\kappa^{-1}|_{S^1}$, denoted by $\phol{-1}$, do extend to $S^1$, and the space of polynomial sections in $H^{-s}(\delt)$, $\phol{-s}$, is mapped into $\mathcal{U}\left(\phol{-1}\right)$ by the embedding from \S 3. Since $\phol{-s}$ is dense in $\hol{-s}$, the action of $\phol{-s}$ on $\Omega^{1/2}(S^1)$ extends to an action of $\hol{-s}$.

\section{The Equivariant Cross-Section of $S^s\left(H^{-1}(\delt)\right)\longrightarrow H^{-s}(\delt)$}

$\ssym$ sits inside $T^s\left(H^{-1}(\delt)\right)$, the space of tensors of order $s$. We will write a monomial in $\ssym$ as
$$\bigodot_{i=1}^sf_i(z)\parsh{z}=\recip{s!}\sum_{\sigma\in S_s}\left(f_{\sigma(1)}(z)\parsh{z}\otimes\ldots\otimes f_{\sigma(s)}(z)\parsh{z}\right),$$
where $S_s$ is the symmetric group on $s$ elements. For each $s$, $S^s\left(H^{-1}(\delt)\right)$ projects onto $H^{-s}(\delt)$ via the map
$$\mathcal{P}_s:\bigodot_{i=1}^sf_i(z)\parsh{z}\longmapsto\prod_{i=1}^sf_i(z)\left(\parsh{z}\right)^s.$$
Let $d_p=z^p\parsh{z}$. We refer to $p$ as the \emph{power} of $d_p$. The vectors 
$$\left\{\bigodot_{i=1}^sd_{p_i}\quad\Bigg|\quad p_1\geq p_2\geq\ldots\geq p_s,\quad p_i\in\N\right\}$$
are a basis for $\ssym$. We refer to $\sum_{i=1}^sp_i$ as the \emph{total power} of such a basis vector.  

The infinitesimal action of $\mathfrak{sl}(2,\C)$ on $H^m(\delt)$, in terms of the coordinate $f$, is
$$A^-=\left(\begin{matrix} 0&0\\
1&0\end{matrix} \right)\longmapsto-\frac{\partial}{\partial z},\qquad
A^+=\left(\begin{matrix} 0&1\\
0&0\end{matrix} \right)\longmapsto z^2\frac{\partial}{\partial z}+2mz,\qquad E=\left(\begin{matrix} 1&0\\
0&-1\end{matrix} \right)\longmapsto 2z\frac {\partial}{\partial z}+2m.$$ 
The action of $\mathfrak{sl}(2,\C)$ on $T^s\left(H^{-1}(\delt)\right)$ is by a Liebniz rule, and preserves the subspace of symmetric tensors. We denote this action by $\pi_{\odot}(X)$. The following lemma is key.

\begin{lem}
$\left(\symraise\right)^{2s}\left(d_0\right)^{\odot s}=C_s\left(d_2\right)^{\odot s}$, for some constant $C_s\in\R$.
\end{lem}

\begin{proof}
The raising operator $\symraise$ maps a monomial of total power $p$ into a linear combination of monomials having total power $p+1$. Thus,
$$\left(\symraise\right)^{2s}\left(d_0\right)^{\odot s}=\sum_{p\in\mathcal{P}^{2s}_s}a_p\bigodot_{i=1}^sd_{p(i)},$$
where $\mathcal{P}^{2s}_s$ is the set of partitions of ${2s}$ into $s$ or fewer parts, arranged so that $p(i)\geq p(i+1)$. At the same time, $A^+d_2=0$, so none of the parts $p(i)$ can be greater than $2$. But there is only one partition of ${2s}$ into $s$ or fewer parts less than or equal to $2$, namely $p(i)=2$ for all $i$. 
\end{proof}

\begin{prop}
\label{equisection}
For each $s$, there is a $PSU(1,1)$-equivariant cross-section of $S^s\left(H^{-1}(\delt)\right)\longrightarrow H^{-s}(\delt)$.
\end{prop}

\begin{proof}
Our strategy is to map the basis elements of $H^{-s}(\delt)$, namely the set $\left\{z^p\left(\parsh{z}\right)^s\right\}_{p=0}^{\oo},$ into $S^s\left(H^{-1}(\delt)\right)$, in such a way that the resulting cross-section is equivariant. 

It is clear from the action of $\symham$ on $\ssym$ that $z^p\left(\parsh{z}\right)^s$ must be mapped to an element having total power $p$. Thus, we map $\left(\parsh{z}\right)^s$ to the only monomial of total power zero, the vector $\paren{d_0}^{\odot s}$. We map $z^{2s+1}\left(\parsh{z}\right)^s$ into a vector $v_{2s+1}$ of total power $2s+1$. The images of all other basis vectors are obtained by applying the raising operator. The resulting cross-section will be equivariant provided that
\begin{align}\label{cond}
\symlow v_{2s+1}=\left(\symraise\right)^{2s}\paren{d_0}^{\odot s}=C_s\left(d_2\right)^{\odot s}.
\end{align}

In fact, we take
$$v_{2s+1}=-\frac{C_s}{3}d_3\odot\left(d_2\right)^{\odot(s-1)}-\sum_{4\leq p\leq2s+1}d_p\odot\sum_{\substack{m+n\leq s-1\\2m+n=2s+1-p}}A_{pn}\paren{d_2}\opow{m}\odot\paren{d_1}\opow{n}\odot\paren{d_0}\opow{\paren{s-m-n-1}}.$$
For example, for $s=2$, one has
$$-\frac{3}{C_2}v_5=d_3\odot d_2-\recip{2}d_4\odot d_1+\recip{10}d_5\odot d_0.$$
The monomials in $v_{2s+1}$ all have total power $2s+1$. The largest power $p$ ranges from $3$ to $2s+1$. Among basis elements with highest power $p$, the sum includes all those where the remaining powers are no greater than two. Thus the inner sum is really over partitions of $2s+1-p$ into $s-1$ or fewer parts, each part being $1$ or $2$. The  coefficients $A_{pn}$ are indexed by the highest power and $n$, the (symmetric tensor) exponent of $d_1$. If $p$ is even, $n$ is odd, and vice-versa. Modulo parity, $n$ ranges from $0$ to the minimum of $p-3$ and $2s+1-p$. Thus the range of $n$ increases until $p=s+2$, and then decreases to $n=0$ when $p=2s+1$. For all other $p,n$, take $A_{pn}=0$.

Because the monomials in $v_{2s+1}$ are all products of $d_p$, $d_2$, $d_1$, and $d_0$, applying $\symlow$ to each monomial results in at most three terms. In fact,
\begin{align*}
\symlow v_{2s+1}=&C_s\left(d_2\right)^{\odot(s)}+\left(\frac{2C_s(s-1)}{3}+4A_{41}\right)d_3\odot\left(d_2\right)^{\odot(s-2)}\odot d_1\\
&+\sum_{4\leq q\leq2s}d_q\odot\sum_{\substack{l+k\leq s-1\\2l+k=2s-q}}B_{qk}\paren{d_2}\opow{l}\odot\paren{d_1}\opow{k}\odot\paren{d_0}\opow{\paren{s-l-k-1}},
\end{align*}
where
$$B_{qk}=(q+1)A_{(q+1)k}+(2s-q-k+2)A_{q(k-1)}+(k+1)A_{q(k+1)}.$$
Thus, choosing $A_{41}=-\frac{C_s(s-1)}{6}$, and defining
\begin{align*}
A_{(p+1)n}=-\recip{p+1}\left[(2s-p-n+2)A_{p(n-1)}+(n+1)A_{p(n+1)}\right]
\end{align*}
for all other $p$ and $n$, one obtains (\ref{cond}). This recursion relation is linear, and terminates at $p=2s+1$. Thus it can be solved, and $v_{2s+1}$ exists and satisfies (\ref{cond}). This proves the Proposition.
\end{proof}

\section{Lowest Weight Vectors in $\linear$}

The action of $\algebra$ on $\linear=\hol{1/2}\otimes\hol{1/2}$ is again by a Liebniz rule, and will be denoted $\pi_{\otimes}(X)$. The irreducible subspaces of $\hol{1/2}\otimes\hol{1/2}$ are lowest-weight representations of $SU(1,1)$. We now identify the vectors in $\hol{1/2}\otimes\hol{1/2}$ annihilated by $\tenslow$.

\begin{prop}
\label{lowestweight}
The set of vectors
$$\left\{l_s:=\sum_{i=0}^{s}(-1)^i\nchoosek{s}{i}z^{s-i}\dzh\otimes z^i\dzh\right\}_{s=0}^{\oo}$$
are annihilated by the operator $\tenslow$. The vector $l_s$ has weight $2(s+1)$.
\end{prop}

\begin{proof}
Let $b_p=z^p\dzh$. Applying $\tenslow$ to $b_{s-i}\otimes b_i$ results in two terms unless $i=0$ or $i=s$, so we pull these cases out of the sum $l_s$. Thus 
\begin{align*}
-\tenslow\left[l_s\right]=&sb_{s-1}\otimes b_0+\sum_{i=1}^{s-1}(-1)^i\nchoosek{s}{i}\tenslow\left[b_{s-i}\otimes b_i\right]+(-1)^{s}sb_0\otimes b_{s-1}.
\end{align*}
The middle term is
\begin{align*}
&\sum_{i=1}^{s-1}(-1)^i\nchoosek{s}{i}\left[(s-i)b_{s-i-1}\otimes b_i+ib_{s-i}\otimes b_{i-1}\right]\\
&=\sum_{i=1}^{s-1}(-1)^i\frac{s\ldots(s-i)}{i!}b_{s-i-1}\otimes b_i+\sum_{i=1}^{s-1}(-1)^i\frac{s\ldots(s-i+1)}{(i-1)!}b_{s-i}\otimes b_{i-1}\\
&=\sum_{i=1}^{s-1}(-1)^i\frac{s\ldots(s-i)}{i!}b_{s-i-1}\otimes b_i+\sum_{j=0}^{s-2}(-1)^{j+1}\frac{s\ldots(s-j)}{j!}b_{s-j-1}\otimes b_j\\
&=-sb_{s-1}\otimes b_0+\sum_{j=1}^{s-2}\left[(-1)^j+(-1)^{j+1}\right]\frac{s\ldots(s-j)}{j!}b_{s-j-1}\otimes b_j+(-1)^{s-1}sb_0\otimes b_{s-1}\\
&=-sb_{s-1}\otimes b_0+(-1)^{s-1}sb_0\otimes b_{s-1}.\\
\end{align*}
Thus,
$$A^-\left[l_s\right]=sb_{s-1}\otimes b_0-sb_{s-1}\otimes b_0+(-1)^{s-1}sb_0\otimes b_{s-1}+(-1)^ssb_0\otimes b_{s-1}=0.$$
To prove the second claim, simply notice that
$$E\left[b_{s-i}\otimes b_i\right]=2(s-i+i+1)b_{s-i}\otimes b_i=2(s+1)b_{s-i}\otimes b_i.$$
\end{proof}

\section{An Explicit Formula for $B_{s+1}$}

We will use the following lemma to prove Theorem \ref{main}. For $x(z)\paren{\parsh{z}}^s\in\hol{-s}$, let
$$\bigo_j\left(x\right)=\pplus x^{(s-j)}\left(\parsh{z}\right)^j\in\linear.$$

\begin{lem}
\label{form}
Let $k>s-j$. As an element of $\hol{1/2}\otimes\hol{1/2}$, $$\bigo_j\left(z^k\right)=\sum_{i=0}^{k-s-1}(-1)^j\frac{(i+j)!}{i!}\frac{k!}{(k-s+j)!}z^{(k-s-1)-i}\dzh\otimes z^i\dzh.$$
\end{lem}

\begin{proof}
Let $f(z)\dzh\in H^{1/2}_{poly}\left(\delt^*\right)$, with $f(z)=\sum_{n=1}^Nf_{-n}z^{-n}$ for $N>k-s$. Then,
$$\left(\parsh{z}\right)^jf(z)=\sum_{n=1}^N(-1)^j\frac{(n+j-1)!}{(n-1)!}f_{-n}z^{-(n+j)}.$$
Also,
$$\left(\parsh{z}\right)^{s-j}z^k=\frac{k!}{(k-s+j)!}z^{k-s+j}.$$
Thus,
$$\left[\left(\parsh{z}\right)^{s-j}z^k\right]\left[\left(\parsh{z}\right)^jf(z)\right]=\sum_{n=1}^N(-1)^j\frac{(n+j-1)!}{(n-1)!}\frac{k!}{(k-s+j)!}f_{-n}z^{k-s-n},$$
and so
$$\bigo_j\left(z^k\right)f(z)\dzh=\sum_{n=1}^{k-s}(-1)^j\frac{(n+j-1)!}{(n-1)!}\frac{k!}{(k-s+j)!}f_{-n}z^{k-s-n}\dzh.$$
Since
$$f_{-n}z^{k-s-n}\dzh=\left(z^{k-s-n}\dzh\otimes z^{n-1}\dzh\right)f(z)\dzh,$$
the required formula is obtained after reindexing. But this formula depends on only the first $k-n$ coefficients of $f$, so it applies to all of $H^{1/2}_{L^2}(\delt^*)$.
\end{proof}

\begin{proof}[Proof of Theorem \ref{main}]
By Proposition \ref{equisection}, we know $B_{s+1}(v)=\pplus L_s(v),$ where
$$L_s(v)=\sum_{j=0}^sc_j(v)\pz{j}.$$
The cases $s=0$ and $s=1$ suggest that $c_j(v)=a_jv^{(s-j)}.$ Since $B_{s+1}$ is unique, we need only find coefficients $a_j$ which satisfy
\begin{align}\label{eqn}\pplus L_s\left(z^{2s+1}\right)=l_s. \end{align}
By Proposition \ref{lowestweight} and Lemma \ref{form}, (\ref{eqn}) is equivalent to
$$\sum_{j=0}^sa_j\sum_{i=0}^s(-1)^j\frac{(i+j)!}{i!}\frac{(2s+1)!}{(s+j+1)!}z^{s-i}\dzh\otimes z^i\dzh=\sum_{i=0}^s(-1)^i\nchoosek{s}{i}z^{s-i}\dzh\otimes z^i\dzh.$$
In other words, the $a_j$'s must solve the linear system
\begin{align*}
M_s\vec{a}_s:=\left(\begin{tabular}{c}
\\
\\
$\Big((-1)^j${\Large$\frac{(i+j)!}{i!}\frac{(2s+1)!}{(s+j+1)!}$}$\Big)_{i,j=0}^s$ \\
\\
\\
\end{tabular}\right)
\left(\begin{tabular}{c}
$a_0$\\
$a_1$\\
$\vdots$\\
$a_s$
\end{tabular}\right)=\left(\begin{tabular}{c}
$1$\\
$-s$\\
$\vdots$\\
$(-1)^s$\\
\end{tabular}\right)=:\vec{l}_s.
\end{align*}
Pleasantly, the matrix $M_s$ is easily factored. First, we rewrite $M_s$ as
$$N_sD_s:=\left(\frac{(i+j)!}{i!j!}\right)_{i,j=0}^s\mathrm{diag}\left((-1)^jj!\frac{(2s+1)!}{(s+j+1)!}\right)_{j=0}^s.$$
Now, a beautiful combinatorial identity comes into play, and one has
$$N_s=\paren{\nchoosek{i+j}{i}}_{i,j=0}^s=\paren{\nchoosek{i}{j}}_{i,j=0}^s\paren{\nchoosek{j}{i}}_{i,j=0}^s=:L_sU_s,$$
using (5.23) from \cite{concmath}. Let
$$\tilde{L}_s=\left((-1)^{i+j}\nchoosek{i}{j}\right)_{i,j=0}^s$$
and let $P=L_s\tilde{L}_s$. Then the $(i,j)$th entry of $P$ is
\begin{align*}
p_{ij}=&\sum_{k=j}^i(-1)^{k+j}\nchoosek{i}{k}\nchoosek{k}{j}=\nchoosek{i}{j}\sum_{k=j}^i(-1)^{k+j}\nchoosek{i-j}{i-k}
=\nchoosek{i}{j}\sum_{k=0}^{i-j}(-1)^k\nchoosek{i-j}{k}=\delta_{ij},
\end{align*}
since $\sum_{k=0}^n(-1)^k\nchoosek{n}{k}=(1-1)^n=\delta_{n0}$. Thus $\tilde{L}_s=L_s^{-1}$. Analogously,
$$U_s^{-1}=\left((-1)^{i+j}\nchoosek{j}{i}\right)_{i,j=0}^s.$$
Now,
\begin{align*}
L_s^{-1}\vec{l}_s=&\left(\sum_{k=0}^s(-1)^{j+k}(-1)^k\nchoosek{j}{k}\nchoosek{s}{k}\right)_{j=0}^s=\left((-1)^j\nchoosek{s+j}{j}\right)_{j=0}^s,
\end{align*}
again using (5.23) from \cite{concmath}. Next,
\begin{align*}
U_s^{-1}L_s^{-1}\vec{l}_s=&\left((-1)^j\sum_{k=0}^s\nchoosek{k}{j}\nchoosek{s+k}{k}\right)_{j=0}^s=\left((-1)^j\nchoosek{s+j}{j}\sum_{k=j}^s\nchoosek{s+k}{s+j}\right)_{j=0}^s\\
=&\left((-1)^j\nchoosek{s+j}{j}\nchoosek{2s+1}{s+j+1}\right)_{j=0}^s,
\end{align*}
using upper summation ((5.10) in \cite{concmath}), and finally,
\begin{align*}
\vec{a}_s=M_s^{-1}l_s=&\left(\frac{(s+j+1)!}{j!(2s+1)!}\nchoosek{s+j}{j}\nchoosek{2s+1}{s+j+1}\right)_{j=0}^s=\left(\recip{s!}\nchoosek{s}{j}\nchoosek{s+j}{j}\right)_{j=0}^s.
\end{align*}
\end{proof}

\section{Binomial Coefficient Identities}

The equivariance of $B_{s+1}$ implies a pair of identities relating sums of products of binomial coefficients.

\begin{prop}
For $s\in\N$, $k\geq 2s+1$, and $l=0,\ldots,k-s$,
\begin{align*}
&\sum_{j=0}^s(-1)^j\nchoosek{s+j}{j}\nchoosek{k}{s-j}\left[\nchoosek{l+j}{j}(k-s)-\nchoosek{l+j-1}{j-1}l\right]\\
&=\sum_{j=0}^s(-1)^j\nchoosek{s+j}{j}\nchoosek{k+1}{s-j}\nchoosek{l+j}{j}(k-2s),
\end{align*}
and for $s\in\N$, $i+j\geq s$,
\begin{align*}
&\sum_{l=0}^s(-1)^l\nchoosek{s+l}{l}\nchoosek{j+l}{l}\nchoosek{i+j+s+1}{s-l}=\sum_{l=0}^s(-1)^l\nchoosek{s}{j}\nchoosek{j}{l}\nchoosek{i}{s-l}.
\end{align*}
\end{prop}

\begin{proof}
To prove the first identity, expand the equation
$$B_{s+1}\left(A^+z^k\right)=\pi_{\otimes}\left(A^+\right)B_{s+1}\left(z^k\right)$$
in the basis $\{z^i\dzh\otimes z^j\dzh\}$; to prove the second, expand the equation
$$B_{s+1}\left(\left(A^+\right)^{(i+j-s)}z^{2s+1}\right)=\left(\pi_{\otimes}\left(A^+\right)\right)^{(i+j-s)}B_{s+1}\left(z^{2s+1}\right).$$
\end{proof}

\section{Connections With Prior Work}

The \emph{transvectant} $\tau_{\frac{1}{2},\frac{1}{2}}^s=\tau_{s+1}$ is the essentially unique equivariant map $\tau_{s+1}:\hol{1/2}\otimes\hol{1/2}\longrightarrow\hol{s+1}$. Let $\theta=f(z)\dzh,\eta=g(z)\dzh\in\hol{1/2}$. Then
$$\tau_{s+1}(\theta\otimes\eta)=\tau_{s+1}(f,g)\dz{s+1}=\left(\sum_{j=1}^s(-1)^j\nchoosek{s}{j}^2f^{(s-j)}g^{(j)}\right)\dz{s+1}.$$
In \cite{JP}, this map is used to construct higher-order Hankel bilinear forms on the space $\hol{1/2}\otimes\hol{1/2}$. If $\nu=x(z)\pz{s}\in\hol{-s}$ and $\mu=y(z)\dz{s+1}\in\hol{s+1}$, then $\bar{\nu}\mu=\bar{x}ydz$ is a one-density on $S^1$, and so $\paren{\hol{s+1}}^*=\hol{-s}$. As in \cite{JP}, define the Hankel form of order $s+1$ with symbol $x$ to be
$$K_{s+1}(x)[f,g]=\int_{S^1}\bar{\nu}\tau_{s+1}(\theta\otimes\eta)=\int_{S^1}\bar{x}\tau_{s+1}(f,g)dz,$$
where $\nu=x(z)\pz{s}$. Since $B_{s+1}(x)\bar{\theta}\in\hol{1/2}$, another bilinear form is defined by
$$\tilde{K}_{s+1}(x)[f,g]=\langle B_{s+1}(x)\bar{\theta},\eta\rangle_{1/2}=\int_{S^1}\overline{B_{s+1}(x)\bar{f}}g d\theta,$$
and $K_{s+1}(x)$ and $\tilde{K}_{s+1}(x)$ are easily seen to be equivalent. Another expression for $\tilde{K}_{s+1}(x)$ is
$$\tilde{K}_{s+1}(x)[f,g]=\left\langle B_{s+1}(x),\theta\otimes\eta\right\rangle_{\otimes}=tr\left(B_{s+1}(x)(\theta\otimes\eta)^*\right),$$
where $\langle,\rangle_{\otimes}$ is the inner product on $\hol{1/2}\otimes\hol{1/2}$. Thus, $B_{s+1}$ is the adjoint of the map $\tau_{s+1}$, and we have the diagram
$$\begin{tabular}{clc}
$\hol{-s}$ & & \\
 & & \\
$I_s\downarrow$ & $\searrow^{B_{s+1}}$ & \\
 & & \\
$\hol{s+1}$ & $\xlongleftarrow{\tau_{s+1}}$ & $\linear$.\end{tabular}$$

\vspace{.4 cm}

\begin{center}
\textsc{Acknowledgements}
\end{center}

\vspace{.3 cm}

\noindent I would like to thank Doug Pickrell for his invaluable assistance, as well as Richard Rochberg, Xiang Tang, and Marcus S\"undhall for helpful conversations.

\end{document}